\documentclass[12pt,a4paper]{article} 
\RequirePackage[margin=1in]{geometry} 

\RequirePackage{commands.2.1}
\RequirePackage{hyperref.2.0}
\RequirePackage{figures.2.0}
\RequirePackage{dots-in-sections.2.0} 
\RequirePackage{english-environments.2.0}

\usepackage{graphicx}
\usepackage{enumitem}
\setlist[enumerate,1]{label=\textup{\arabic*)}}

\RequirePackage{bm} 

\DeclareMathOperator{\Span}{span}

\begin{document}
	\selectlanguage{english}
	
	\title{A closed-form convergence criterion\\for the weak greedy algorithm}
	\author{Mikhail Novikov\thanks{The work was supported by the Ministry of Science and Higher Education of the Russian Federation (agreement no. 075–15–2025–343).}}
	\date{12.06.2026}
	
	\maketitle
	
	\begin{abstract}
		In 2002, V. N. Temlyakov established a criterion for the convergence of the weak greedy algorithm in a Hilbert space for a given weakness sequence $ \tau = \{t_1,t_2,\ldots\} $. The criterion requires verifying a certain limiting relation for every nonnegative square-summable sequence. We give an equivalent closed-form criterion: the weak greedy algorithm converges if and only if~$ \sum_{n=1}^{\infty}(1+ n\sum_{k=1}^{n}t_k^2 )^{-1/2}t_n^2=+\infty $.
	\end{abstract}

	

\section{Introduction}
In this paper, we derive a closed-form criterion for the convergence of the weak greedy algorithm (WGA) expressed solely in terms of the weakness sequence $ \tau $ (see Theorem~\ref{2026-05-17_13-33-16}). Background on this algorithm can be found in \cite{On_the_convergence_of_a_weak_greedy_algorithm},
\cite{Weak_greedy_algorithms}, and
\cite{A_Criterion_for_Convergence_of_Weak_Greedy_Algorithms}. The central result of these papers is a convergence criterion for this algorithm (see Theorem~\ref{2026-05-17_13-42-08}). In Theorem~\ref{2026-05-17_13-33-16}, we present simpler form of this condition. We begin by describing the algorithm itself and explaining what we mean by its convergence.

Let $ H $ be a Hilbert space, and let $ \Dc $ be a dictionary. This means that $ \overline{\Span}\,\Dc=H $ and $ \Dc\subset\{g\in H\colon \norm{g}{}=1 \} $. Fix a weakness sequence $ \tau = \{t_k\}_{k=1}^{\infty} $, $ 0\le t_k\le 1 $, and consider an arbitrary element $ f\in H $. The weak greedy algorithm recursively constructs the objects $ \phi_{n}^{\tau}\in\Dc $, $ f_n^{\tau}\in H $, and $ G_n^{\tau}(f,\Dc) $ for all $ n\in\N $. This is done as follows. Set $ f_0^{\tau}=f $ and choose, for each $ n\in\N $, an element $ \phi_n^{\tau}\in\Dc $ such that
\begin{gather}
	\notag
	\LB| \LB\la f_{n-1}^{\tau},\phi_n^{\tau} \RB\ra \RB|
	\ge
	t_n\sup_{g\in\Dc}\LB\{ \LB| \LB\la f_{n-1}^{\tau}, g \RB\ra \RB| \RB\}.
\end{gather}
Next define
\begin{gather}
	\notag
	f_n^{\tau}
	=
	f_{n-1}^{\tau}
	-
	\LB\la f_{n-1}^{\tau},\phi_n^{\tau} \RB\ra\phi_n^{\tau},
	\quad
	G_n^{\tau}(f,\Dc)
	=
	\sum_{j=1}^{n}\LB\la f_{j-1}^{\tau}, \phi_j^{\tau} \RB\ra\phi_j^{\tau}.
\end{gather}
The sequence $ G_n^{\tau}(f,\Dc) $ is the output of the weak greedy algorithm. The problem studied in \cite{On_the_convergence_of_a_weak_greedy_algorithm}, \cite{Weak_greedy_algorithms}, and \cite{A_Criterion_for_Convergence_of_Weak_Greedy_Algorithms} was to determine for which sequences $ \tau $ this algorithm is guaranteed to converge, that is $ G_n^{\tau}(f,\Dc)\to f $ for any $ f\in H $.
\begin{Def}
	For a given sequence $ \tau=\{t_k\}_{k=1}^{\infty} $, $ t_k\in[0,1] $, we say that the weak greedy algorithm converges (WGA converges) if, for every Hilbert space $ H $, every dictionary $ \Dc $, every $ f\in H $, and every possible choice of the vectors $ \phi_n^{\tau} $, the convergence $ G_n^{\tau}(f,\Dc)\to f $ holds in $ H $.
\end{Def}

We state the convergence criterion for the WGA established in the cited papers.

\begin{Th}[Theorems 1.1, 1.2 of \cite{A_Criterion_for_Convergence_of_Weak_Greedy_Algorithms}]
	\label{2026-05-17_13-42-08}
	Consider a sequence $ \tau = \LB\{ t_j \RB\}_{j=1}^{\infty} $\textup, $ t_j\in[0,1] $. The following conditions are equivalent\textup:
	\begin{enumerate}
		\item \label{2026-05-17_13-13-48} WGA converges\textup;
		\item \label{2026-05-17_13-13-46} $ \displaystyle
			\liminf_{n\to\infty}\frac{a_n}{t_n}\sum_{j=1}^{n}a_j = 0
		$ for any $ \ell^2 $-sequence $ a = \LB\{ a_j \RB\}_{j=1}^{\infty} $, $ a_j \ge 0 $\textup;
		\item \label{2026-06-01_12-05-57} $ 
			\displaystyle\sum_{s=1}^{\infty}\frac{2^s}{q_s-q_{s-1}}=+\infty
		$ or $ 
			\displaystyle\sum_{s=1}^{\infty}2^{-s}\sum_{k=1}^{q_s}t_k^2=+\infty
		$ for any integers $ 0=q_0<q_1<q_2\ldots $
	\end{enumerate}
\end{Th}

The main result of this paper is an explicit form of conditions~\ref{2026-05-17_13-13-46}, \ref{2026-06-01_12-05-57} in Theorem~\ref{2026-05-17_13-42-08}.

\begin{Th}
	\label{2026-05-17_13-33-16}
	Let $ \tau = \LB\{ t_j \RB\}_{j=1}^{\infty} $, $ t_j\in[0,1] $. The following conditions are equivalent\textup:
	\begin{enumerate}
		\item \label{2026-05-17_13-35-27} WGA converges\textup;
		\item \label{2026-06-08_16-57-27} The series
		$ \displaystyle 
			\sum_{n=0}^{\infty}\iiiLB( 2^{-n}\sum_{k=2^n}^{2^{n+1}-1}t_k^2 \iiiRB)^{1/2}
		$ diverges\textup;
		\item \label{2026-05-17_13-35-30} The series
		$ \displaystyle
			\sum_{n=1}^{\infty}\iiiLB( 1 + n\sum_{k=1}^{n}t_k^2 \iiiRB)^{-1/2}t_n^2
		$ diverges\textup;
		\item \label{2026-06-09_11-22-20} The series
		$ \displaystyle 
			\sum_{n=1}^{\infty}\frac{1}{\sqrt{n}}\iiiLB(  
				\iiiLB( \sum_{k=1}^{n}t_k^2 \iiiRB)^{1/2}
				\!\!-
				\iiiLB( \sum_{k=1}^{n-1}t_k^2 \iiiRB)^{1/2}\,
			\iiiRB)
		$ diverges.
	\end{enumerate}
\end{Th}

\begin{Rem}
	\label{2026-05-17_14-35-03}
	Condition~\ref{2026-05-17_13-35-30} of Theorem~\ref{2026-05-17_13-33-16} can be rewritten in the following way:
	\begin{gather}
		\label{2026-05-17_13-20-54}
		\sum_{n=1}^{\infty}\iiiLB( n\sum_{k=1}^{n}t_k^2 \iiiRB)^{-1/2}t_n^2=+\infty.
	\end{gather}
	In this form, the expression is homogeneous with respect to the sequence $ \tau $. But then
	we should clarify that $ \iiLB( n\sum_{k=1}^{n}t_k^2 \iiRB)^{-1/2}t_n^2 = 0 $ whenever $ t_1=t_2=\ldots=t_n=0 $.
\end{Rem}

\subsection*{Remarks and outline of the paper}
In Section~\ref{2026-05-26_03-09-24}, we review previous results. Note that condition~\ref{2026-06-08_16-57-27} of Theorem~\ref{2026-05-17_13-33-16} was known earlier. Namely, it was proved in Theorem~3 of~\cite{On_the_convergence_of_a_weak_greedy_algorithm} that condition~\ref{2026-06-08_16-57-27} implies the convergence of WGA. We prove the reverse implication in Lemma~\ref{2026-06-08_21-18-46}. In fact, this is sufficient to obtain an explicit convergence criterion for the WGA.

In Section~\ref{2026-06-09_17-29-29} we prove Theorem~\ref{2026-05-17_13-33-16} and in Lemma~\ref{2026-06-09_15-53-36} we show that related expressions are equal up to multiplicative constants. In Lemma~\ref{2026-06-09_15-53-36} we consider a generalization of the sum $ \sum_{n=0}^{\infty}\iiLB( 2^{-n}\sum_{k=2^n}^{2^{n+1}-1}t_k^2 \iiRB)^{1/2} $, which was originally proposed in the article~\cite{New_Conditions_for_the_Convergence_of_a_Weak_Greedy_Algorithm}.

Section~\ref{2026-06-09_17-30-14} is devoted to explaining what these expressions mean. We show that the sum $ \sum_{n=1}^{\infty}(n\sum_{k=1}^{n}t_k^2)^{-1/2}t_n^2 $ is equal up to a constant to the minimal value of the sum $ \sum_{n=1}^{\infty}a_n^2 $ in which numbers $ a_n $ satisfy the inequality $ a_n\sum_{k=1}^{n}a_k\ge t_n $. Initially, an explicit criterion for convergence of the WGA was found while solving this optimization problem related to condition~\ref{2026-05-17_13-13-46} of Theorem~\ref{2026-05-17_13-42-08}. Later, with the help of ChatGPT-5.5 Thinking, we observed that the newly obtained condition~\ref{2026-05-17_13-35-30} and the previously known condition~\ref{2026-06-08_16-57-27} of Theorem~\ref{2026-05-17_13-33-16} are equivalent. This allowed us to significantly simplify the proof of the main result. Nevertheless, we have decided to keep the previous proof of the main result in order to illustrate the meaning of the obtained expression.


\section{Summary of previous results}
\label{2026-05-26_03-09-24}
In this section, we review previously obtained partial results concerning this problem. Specifically, we present the necessary and sufficient conditions for the convergence of the weak greedy algorithm that are expressed solely in terms of the sequence $ \tau $ and were known prior to this article. We also mention several special cases of the problem. We start with sufficient conditions for convergence of the WGA.
\begin{Lm}[Th. 1 of \cite{Weak_greedy_algorithms}, Th. 3 of \cite{On_the_convergence_of_a_weak_greedy_algorithm}]
	\label{2026-05-26_12-37-30}
	Let $ \tau = \{t_k\}_{k=1}^{\infty} $, $ t_k\in[0,1] $. Then
	\begin{gather}
		\label{2026-06-09_17-39-54}
		\sum_{n=1}^{\infty}\frac{t_n}{n}=+\infty
		\qRarr
		\sum_{n=0}^{\infty}\iiiLB( 2^{-n}\sum_{k=2^n}^{2^{n+1}-1}t_k^2 \iiiRB)^{1/2}=+\infty
		\qRarr
		\text{WGA converges}.
	\end{gather}
\end{Lm}
A generalization of the second implication can be found in Theorem~9 of~\cite{New_Conditions_for_the_Convergence_of_a_Weak_Greedy_Algorithm}. As mentioned above, the second implication of~\eqref{2026-06-09_17-39-54} in fact can be reversed. The first implication is not reversible. As a counterexample, one can take $ t_n = 1 $ if $ n = k\lfloor\ln(k)\rfloor^2 $ for some $ k\in\N $ and $ t_n = 0 $ otherwise\footnote{Here $ \lfloor x\rfloor $ denotes the largest integer $ m $
such that $ m\le x $.} (see Remark~\ref{2026-06-05_14-43-02}). The following lemma formulates the previously known necessary conditions for the convergence of the WGA.

\begin{Lm}[Rem. 2 of \cite{Weak_greedy_algorithms}, Th. 7 of \cite{New_Conditions_for_the_Convergence_of_a_Weak_Greedy_Algorithm}]
	\label{2026-05-26_12-54-22}
	Let $ \tau = \{t_k\}_{k=1}^{\infty} $, $ t_k\in[0,1] $. Then
	\begin{gather}
		\notag
		\text{WGA converges}
		\qRarr
		\sum_{n=1}^{\infty}\iiiLB(1 + \sum_{k=1}^{n}t_k \iiiRB)^{-1}t_n^2=+\infty
		\qRarr
		\sum_{n=1}^{\infty}t_n^2=+\infty.
	\end{gather}
\end{Lm}

These conditions are not sufficient. For example, we can choose $ t_n = 1 $ if $ n=k^2 $ and $ t_n = 0 $ otherwise. Note that the expressions in Lemmas~\ref{2026-05-26_12-37-30} and~\ref{2026-05-26_12-54-22} provide lower and upper bounds for the expressions in Theorem~\ref{2026-05-17_13-33-16}. In particular, Lemmas~\ref{2026-05-26_12-37-30} and \ref{2026-05-26_12-54-22} follow from Theorem~\ref{2026-05-17_13-33-16} and the next lemma which we state without proof.

\begin{Lm}
	\label{2026-06-10_10-58-50}
	Let $ \tau = \{t_k\}_{k=1}^{\infty} $, $ t_k\ge 0 $. Then
	\begin{gather}
		\notag
		\sum_{n=1}^{\infty}\frac{t_n}{n}
		\le
		\sum_{n=0}^{\infty}\iiiLB( 2^{-n}\sum_{k=2^n}^{2^{n+1}-1}t_k^2 \iiiRB)^{1/2}
		\asymp
		\sum_{n=1}^{\infty}\iiiLB( n\sum_{k=1}^{n}t_k^2 \iiiRB)^{-1/2}t_n^2
		\le
		\sum_{n=1}^{\infty}\iiiLB(\sum_{k=1}^{n}t_k \iiiRB)^{-1}t_n^2.
	\end{gather}
\end{Lm}

The next lemma shows that the convergence of the WGA is equivalent to the divergence of the series $ \sum_{n=1}^{\infty}t_n/n $ if the sequence $ \tau $ is sufficiently regular.

\begin{Lm}[Th. 8 of \cite{New_Conditions_for_the_Convergence_of_a_Weak_Greedy_Algorithm}]
	\label{2026-05-27_05-54-31}
	Let $ \tau = \{t_k\}_{k=1}^{\infty} $, $ t_k\in[0,1] $, and suppose that there are a sequence $ \{\lambda_k\}_{k=1}^{\infty} $\textup, $ \lambda_k> 0 $\textup, and a constant $ c>0 $ such that
	\begin{gather}
		\label{2026-05-27_06-03-34}
		\sum_{k=1}^{n}\lambda_k
		\ge
		c\,n\lambda_n
		\qtext{and}
		\frac{t_n}{\lambda_n}
		\ge
		\frac{t_{n+1}}{\lambda_{n+1}}
		\quad
		\text{for all }
		n\in\N.
	\end{gather}
	Then WGA converges if and only if $ \displaystyle 
		\sum_{n=1}^{\infty}\frac{t_n}{n}=+\infty
	$.
\end{Lm}

As a corollary, we obtain the following statement, which was first formulated as Theorem~1 in \cite{On_the_convergence_of_a_weak_greedy_algorithm}.

\begin{Cor}
	Let $ \tau=\{t_k\}_{k=1}^{\infty} $, $ 1\ge t_k\ge t_{k+1}\ge 0 $ for all $ k\in\N $. Then
	\begin{gather}
		\notag
		\text{WGA converges}
		\qLRarr
		\sum_{n=1}^{\infty}\frac{t_n}{n}=+\infty.
	\end{gather}
\end{Cor}

Previous papers also studied a convergence criterion for the weak greedy algorithm in the case when $ t_k\in\{0,1\} $ for every $ k\in\N $.

\begin{Lm}[Th. 2 and Rem. 3.1 of \cite{On_the_convergence_of_a_weak_greedy_algorithm}]
	\label{2026-05-29_08-41-12}
	Let $ \{n_k\}_{k=1}^{\infty} $ be a strictly increasing sequence of natural numbers and let $ \tau=\{t_n\}_{n=1}^{\infty} $ be given by
	\begin{gather}
		\notag
		t_n
		=
		\begin{cases}
			1, & \text{if $ n=n_k $ for some $ k\in\N $} ;\\
			0, & \text{otherwise} .\\
		\end{cases}
	\end{gather}
	If, in addition, there exists $ C>0 $ such that $ n_{k+1}/n_{k}\le C $ for all $ k\in\N $, then
	\begin{gather}
		\label{2026-05-26_02-06-07}
		\sum_{k=1}^{\infty}\frac{(n_{k+1} - n_k)^{1/2}}{n_k}=+\infty
		\qRarr
		\text{WGA converges}.
	\end{gather}
	If, instead, we assume that $ n_{k+2}-n_{k+1}\ge n_{k+1}-n_{k} $ for all $ k\in\N $ then the reverse implication in~\eqref{2026-05-26_02-06-07} holds.
\end{Lm}

\begin{Rem}
	\label{2026-06-05_14-43-02}
	In this setting WGA converges if and only if $ \sum_{k=1}^{\infty}(kn_k)^{-1/2}=+\infty $. This follows directly from Theorem~\ref{2026-05-17_13-33-16}. Hence, the condition $ \sum_{k=1}^{\infty}(n_{k+1}-n_k)^{1/2}/n_k=+\infty $ is neither sufficient nor necessary in the general case. For example, we can choose $ n_k = 2^{2^k} $ and $ \{n_k\}_{k=1}^{\infty}=\bigcup_{m\in\N}\{2^mm^2,\ldots, 2^m(m^2 + 1)\} $ respectively.
\end{Rem}


\section{Proof of the main result}
\label{2026-06-09_17-29-29}
In this section, we prove Theorem~\ref{2026-05-17_13-33-16}. The implication from condition~\ref{2026-06-08_16-57-27} to condition~\ref{2026-05-17_13-35-27} was proved in Theorem~3 of \cite{On_the_convergence_of_a_weak_greedy_algorithm}. In the next lemma, we verify the implication in the other direction.
\label{2026-06-08_21-18-16}
\begin{Lm}
	\label{2026-06-08_21-18-46}
	Let $ \tau = \{t_k\}_{k=1}^{\infty} $, $ t_k\ge 0 $ be such that
	\begin{gather}
		\label{2026-06-08_21-42-06}
		\sum_{m=0}^{\infty}\iiiLB( 2^{-m}\sum_{n=2^m}^{2^{m+1}-1}t_n^2 \iiiRB)^{1/2} <+\infty.
	\end{gather}
	Then there exists an $ \ell^2 $-sequence $ a=\{a_k\}_{k=1}^{\infty} $ such that
	\begin{gather}
		\label{2026-06-08_22-00-53}
		a_n\sum_{k=1}^{n}a_k\ge t_n
		\quad
		\text{for all}
		\ 
		n\in\N.
	\end{gather}
\end{Lm}

\begin{proof}
	For all numbers $ m,n\in\N $ such that $ 2^m\le n\le 2^{m+1}-1 $ define values
	\begin{gather}
		\notag
		x_0 = t_1,
		\quad
		x_m^2
		=
		\sum_{k=2^m}^{2^{m+1}-1}t_k^2;
		\quad
		a_1 = \sqrt{t_1},
		\quad
		a_n
		=
		\max\iiiLB(  
			\frac{t_n}{2^{m/4}\sqrt{x_m}},
			\frac{\sqrt{x_{m+1}}}{2^{(3m-1)/4}}
		\iiiRB).
	\end{gather}
	First, we verify that $ a\in\ell^2 $:
	\begin{multline}
		\notag
		\sum_{n=2}^{\infty}a_n^2
		=
		\sum_{m=1}^{\infty}\sum_{n=2^m}^{2^{m+1}-1}a_n^2
		\le
		\sum_{m=0}^{\infty}\sum_{n=2^m}^{2^{m+1}-1}\LB(  
			\frac{t_n^2}{2^{m/2}x_m} + \frac{x_{m+1}}{2^{(3m-1)/2}}
		\RB)
		\\
		\notag
		=
		\sum_{m=0}^{\infty}2^{-m/2}x_m + 2^{-(m-1)/2}x_{m+1}
		\le
		3\sum_{m=0}^{\infty}2^{-m/2}x_m
		\overset{\eqref{2026-06-08_21-42-06}}{<}
		+\infty.
	\end{multline}
	Since the inequality~\eqref{2026-06-08_22-00-53} holds for $ n=1 $, it suffices to prove this estimate for all $ n,m\in\N $ such that $ 2^{m}\le n\le 2^{m + 1} - 1 $:
	\begin{gather}
		\notag
		a_n\sum_{k=1}^{n}a_k
		\ge
		\frac{t_n}{2^{m/4}\sqrt{x_m}}\sum_{k=2^{m-1}}^{2^m-1}\frac{\sqrt{x_m}}{2^{(3m-4)/4}}
		=
		t_n.
	\end{gather}
\end{proof}

The next lemma shows that conditions \ref{2026-06-08_16-57-27}, \ref{2026-05-17_13-35-30}, \ref{2026-06-09_11-22-20} of Theorem~\ref{2026-05-17_13-33-16} are equivalent and thus finishes the proof of Theorem~\ref{2026-05-17_13-33-16}.

\begin{Lm}
	\label{2026-06-09_15-53-36}
	Let $ \tau = \{t_k\}_{k=1}^{\infty} $, $ t_k\ge 0 $, and suppose that a sequence of natural numbers $ \gamma = \{\gamma_m\}_{m=0}^{\infty} $, $ \gamma_0 = 1 $, satisfies the relations $ \alpha\le \gamma_{m+1}/\gamma_m \le\beta $, where $ 1<\alpha\le\beta $. Then
	\begin{gather}
		\label{2026-06-09_12-15-19}
		\sum_{m=0}^{\infty}\iiiLB( \gamma_m^{-1}\sum_{k=\gamma_m}^{\gamma_{m+1}-1}t_k^2 \iiiRB)^{\frac{1}{2}}
		\ 
		\asymp
		\ 
		\sum_{n=1}^{\infty}\iiiLB( n\sum_{k=1}^{n}t_k^2 \iiiRB)^{-\frac{1}{2}}t_n^2
		\ 
		\asymp
		\ 
		\sum_{n=1}^{\infty}\frac{1}{\sqrt{n}}\iiiLB(  
			\iiiLB( \sum_{k=1}^{n}t_k^2 \iiiRB)^{\frac{1}{2}}
			\!\!-
			\iiiLB( \sum_{k=1}^{n-1}t_k^2 \iiiRB)^{\frac{1}{2}}\,
		\iiiRB),
	\end{gather}
	where the implicit constants in these relations depend on $ \alpha $ and $ \beta $ only.
\end{Lm}

\begin{Rem}
	The inequalities~\eqref{2026-06-09_12-14-28} below are more precise versions of the relations~\eqref{2026-06-09_12-15-19}.
\end{Rem}

\begin{proof}
	Introduce the quantities $ s_n $, $ x_m $, $ y_m $, and $ \lambda $:
	\begin{gather}
		\notag
		s_n^2
		=
		\sum_{k=1}^{n}t_k^2,
		\quad
		x_m^2
		=
		\sum_{k=\gamma_m}^{\gamma_{m+1}-1}t_k^2,
		\quad
		y_m^2
		=
		\sum_{n=0}^{m}x_n^2
		=
		\sum_{k=1}^{\gamma_{m+1}-1}t_k^2;
		\quad
		\lambda
		=
		\sum_{m=0}^{\infty}\frac{x_m^2}{\sqrt{\gamma_m}\, y_m}.
	\end{gather}
	We also set $ s_0 = y_{-1} = 0 $. It suffices to prove the following chain of inequalities:
	\begin{gather}
		\label{2026-06-09_12-14-28}
		\beta^{-1/2}\sqrt{1-\alpha^{-1/2}}\sum_{m=0}^{\infty}\frac{x_m}{\sqrt{\gamma_m}}
		\ 
		\le
		\ 
		\sum_{n=1}^{\infty}\frac{t_n^2}{\sqrt{n}\,s_n}
		\ 
		\le
		\ 
		2\sum_{n=1}^{\infty}\frac{s_n-s_{n-1}}{\sqrt{n}}
		\ 
		\le
		\ 
		2\sum_{m=0}^{\infty}\frac{x_m}{\sqrt{\gamma_m}}.
	\end{gather}
	First, we estimate $ \lambda $ from above:
	\begin{gather}
		\notag
		\lambda
		=
		\sum_{m=0}^{\infty}\sum_{n=\gamma_m}^{\gamma_{m+1}-1}
		\iiiLB(  
			\gamma_m\sum_{k=1}^{\gamma_{m+1}-1}t_k^2
		\iiiRB)^{-1/2}t_n^2
		\le
		\sqrt{\beta}\sum_{n=1}^{\infty}\iiiLB( n\sum_{k=1}^{n}t_k^2 \iiiRB)^{-1/2}t_n^2
		=
		\sqrt{\beta}\sum_{n=1}^{\infty}\frac{t_n^2}{\sqrt{n}\,s_n}.
	\end{gather}
	Hence, the first inequality in~\eqref{2026-06-09_12-14-28} follows from the next relations:
	\begin{gather}
		\notag
		\iiiLB( \sum_{m=0}^{\infty}\frac{x_m}{\sqrt{\gamma_m}} \iiiRB)^{2}
		=
		\iiiLB(  
			\sum_{m=0}^{\infty}\frac{x_m}{\gamma_m^{1/4}\sqrt{y_m}}\cdot\frac{\sqrt{y_m}}{\gamma_m^{1/4}}
		\iiiRB)^{2}
		\le
		\lambda\sum_{m=0}^{\infty}\frac{y_m}{\sqrt{\gamma_m}}
		\\
		\notag
		=
		\lambda\sum_{m=0}^{\infty}\sum_{n=0}^{m}\frac{x_n^2}{\sqrt{\gamma_m}\,y_m}
		\le
		\lambda\sum_{n=0}^{\infty}\frac{x_n^2}{\sqrt{\gamma_n}\, y_n}
		\sum_{m=n}^{\infty}\sqrt{\gamma_n/\gamma_m}
		\le
		\lambda^2\sum_{k=0}^{\infty}\alpha^{-k/2}
		=
		\frac{\lambda^2}{1 - \alpha^{-1/2}}.
	\end{gather}
	The second inequality in~\eqref{2026-06-09_12-14-28} can be obtained as follows:
	\begin{gather}
		\notag
		\sum_{n=1}^{\infty}\frac{t_n^2}{\sqrt{n}\,s_n}
		=
		\sum_{n=1}^{\infty}\frac{s_n^2 - s_{n-1}^{2}}{\sqrt{n}\,s_n}
		\le
		2\sum_{n=1}^{\infty}\frac{s_n - s_{n-1}}{\sqrt{n}}.
	\end{gather}
	It remains to verify the last inequality:
	\begin{multline}
		\notag
		\sum_{n=1}^{\infty}\frac{s_n-s_{n-1}}{\sqrt{n}}
		=
		\sum_{m=0}^{\infty}\sum_{n=\gamma_m}^{\gamma_{m+1}-1}\frac{s_n-s_{n-1}}{\sqrt{n}}
		\le
		\sum_{m=0}^{\infty}\frac{1}{\sqrt{\gamma_m}}\sum_{n=\gamma_m}^{\gamma_{m+1}-1}(s_n-s_{n-1})
		\\
		\notag
		=
		\sum_{m=0}^{\infty}\frac{y_m - y_{m-1}}{\sqrt{\gamma_m}}
		\le
		\sum_{m=0}^{\infty}\frac{\sqrt{y_{m\vphantom{1}}^2 - y_{m-1}^2}}{\sqrt{\gamma_m}}
		=
		\sum_{m=0}^{\infty}\frac{x_m}{\sqrt{\gamma_m}}.
	\end{multline}
\end{proof}


\section{Interpretation of the main expression}
\label{2026-06-09_17-30-14}

The following theorem explains what the expressions in Theorem~\ref{2026-05-17_13-33-16} mean in the case where they are finite.

\begin{Th}
\label{2026-06-09_15-06-45}
	Let $ \tau = \{t_k\}_{k=1}^{\infty} $\textup, $ t_k\in[0,1] $, $ t_1>0 $. Then
	\begin{gather}
		\notag
		\inf\iiiLB\{  
			\sum_{n=1}^{\infty}a_n^2\colon
			a_n\sum_{k=1}^{n}a_k\ge t_n
		\iiiRB\}
		\asymp
		\inf\iiiLB\{  
			\sum_{n=1}^{\infty}\iiiLB( \sum_{k=1}^{n}u_k \iiiRB)^{-1}u_n^2\colon u_n\ge t_n
		\iiiRB\}
		\asymp
		\sum_{n=1}^{\infty}\iiiLB( n\sum_{k=1}^{n}t_k^2 \iiiRB)^{-1/2}t_n^2.
	\end{gather}
\end{Th}

\begin{Rem}
	This theorem implies that condition~\ref{2026-05-17_13-13-46} of Theorem~\ref{2026-05-17_13-42-08} is equivalent to condition~\ref{2026-05-17_13-35-30} of Theorem~\ref{2026-05-17_13-33-16} and, as a consequence, the WGA converges if and only if the sum $ \sum_{n=1}^{\infty}\LB( n\sum_{k=1}^{n}t_k^2 \RB)^{-1/2}t_n^2 $ diverges. To see this, it suffices to note that if there exists a sequence $ a=\{a_k\}_{k=1}^{\infty} $ such that $ \liminf_{n\to\infty}t_n^{-1}a_n\sum_{k=1}^{n}a_k>0 $ then there exists an $ \ell^2 $-sequence $ \tilde a $ such that $ t_n^{-1}\tilde a_n\sum_{k=1}^{n}\tilde a_k\ge 1 $ for all $ n\in\N $. Indeed, one can choose a sufficiently large number $ m\in\N $ and set
	$ \tilde a_n = 1 $ for $ n\le m $, $ \tilde a_n = m a_n $ for $ n\ge m $.
\end{Rem}

The first relation in Theorem~\ref{2026-06-09_15-06-45} follows from the next lemma.

\begin{Lm}
	\label{2026-06-05_13-47-14}
	Let $ a=\{a_k\}_{k=1}^{\infty} $, $ u=\{u_k\}_{k=1}^{\infty} $, and $ A_n = \sum_{k=1}^{n}a_k $, $ U_n=\sum_{k=1}^{n}u_k $. Assume that $ a_1,u_1>0 $ and for all $ n\in\N $ we have $ a_n,u_n\ge 0 $,
	\begin{gather}
		\label{2026-05-17_15-02-57}
		a_nA_n
		=
		u_n.
	\end{gather}
	Then for all $ n\in\N $ we have
	\begin{gather}
		\label{2026-06-05_11-43-05}
		\frac{u_n}{\sqrt{\mathstrut 2U_n}}
		\le
		a_n
		\le
		\frac{2u_n}{\sqrt{\mathstrut U_n}}.
		\qquad
		\text{In particular\textup,}
		\quad
		\frac{1}{2}\sum_{n=1}^{\infty}\frac{u_n^2}{U_n}
		\le
		\sum_{n=1}^{\infty}a_n^2
		\le
		4\sum_{n=1}^{\infty}\frac{u_n^2}{U_n}.
	\end{gather}
	Moreover, the sequence $ a $ is uniquely determined by the sequence $ u $ and vice versa.
\end{Lm}

\begin{proof}
	First, we define values $ A_0 = U_0 = 0 $ and note that $ A_n\ge a_1>0 $, $ U_n\ge u_1>0 $ for all $ n\in\N $. What is more, inequalities $ a_1>0 $ and $ u_1>0 $ hold simultaneously, because $ a_1^2 = u_1 $. Next, we express $ a_n $ in terms of $ u_n $ and $ A_{n-1} $:
	\begin{gather}
		\notag
		a_n(A_{n-1} + a_n)
		\overset{\eqref{2026-05-17_15-02-57}}{=}
		u_n;
		\quad
		a_n^2+A_{n-1}a_n - u_n = 0;
		\quad
		a_n
		=
		\frac{\pm\sqrt{\mathstrut A_{n-1}^{2} + 4u_n} - A_{n-1}}{2}.
	\end{gather}
	Choosing the nonnegative root, we obtain
	\begin{gather}
		\label{2026-05-17_15-09-45}
		a_n
		=
		\frac{\sqrt{\mathstrut A_{n-1}^2 + 4u_{n}} - A_{n-1}}{2}
		=
		\frac{2u_n}{\sqrt{\mathstrut A_{n-1}^2 + 4u_n} + A_{n-1}}.
	\end{gather}
	Thus, the sequences $ a $ and $ u $ can be uniquely recovered from each other. We are ready to verify inequalities~\eqref{2026-06-05_11-43-05}. For $ n\ge 2 $ the following relations hold:
	\begin{gather}
		\label{2026-05-18_10-28-05}
		a_nA_{n-1}
		\overset{\eqref{2026-05-17_15-09-45}}{=}
		\frac{2u_n}{\sqrt{\mathstrut 1 + 4u_n/A_{n-1}^2} + 1}
		\le
		u_n.
	\end{gather}
	Next, we estimate $ A_n^2 $:
	\begin{gather}
		\label{2026-05-18_10-30-21}
		A_n^2
		=
		A_1^2 + \sum_{k=2}^{n}(A_k^2 - A_{k-1}^2)
		=
		a_1A_1 + \sum_{k=2}^{n}a_k(A_k + A_{k-1})
		=
		\sum_{k=1}^{n}a_kA_k + \sum_{k=2}^{n}a_kA_{k-1}.
	\end{gather}
	Then, by relations~\eqref{2026-05-17_15-02-57} and~\eqref{2026-05-18_10-28-05}, we have the inequalities
	$
		\notag
		A_n^2
		\le
		\sum_{k=1}^{n}u_k + \sum_{k=2}^{n}u_k\le 2U_n
	$.
	On the other hand, there is a lower estimate:	
	$
		\notag
		A_{n}^2
		\ge
		\sum_{k=1}^{n}a_kA_k
		=
		\sum_{k=1}^{n}u_k
		=
		U_n
	$.
	Therefore, we have $ U_n\le A_n^2\le 2U_n $. This yields desired estimates of $ a_n $:
	\begin{align}
		\notag
		a_n
		&\overset{\eqref{2026-05-17_15-09-45}}{=}
		\frac{2u_n}{\sqrt{\mathstrut A_{n-1}^2+4u_n}+A_{n-1}}
		\le
		\frac{2u_n}{\sqrt{\mathstrut U_{n-1}+u_n}}
		=
		\frac{2u_n}{\sqrt{\mathstrut U_n}};
		\\
		\notag
		a_n
		&\ge
		\frac{2u_n}{\sqrt{\mathstrut 2U_{n-1}+4u_n}+\sqrt{\mathstrut 2U_{n-1}}}
		=
		\frac{u_n}{\sqrt{\mathstrut 2U_n}}
		\cdot
		\frac{2}{\sqrt{1 + u_n / U_n}+\sqrt{1 - u_n/U_n}}
		\ge
		\frac{u_n}{\sqrt{\mathstrut 2U_n}}.
	\end{align}
	These expressions make sense because $ U_n>0 $.
\end{proof}

We now prove the second relation in Theorem~\ref{2026-06-09_15-06-45}.

\begin{Lm}
	\label{2026-06-02_12-10-50}
	For any sequence $ \tau = \{t_k\}_{k=1}^{\infty} $\textup, $ t_n\in[0,1] $\textup, the following inequalities hold\textup:
	\begin{gather}
		\label{2026-06-09_13-52-09}
		\frac{1}{2}\sum_{n=1}^{\infty}\iiiLB( n\sum_{k=1}^{n}t_k^{2} \iiiRB)^{-1/2}t_n^{2}
		\le
		\inf\iiiLB\{\sum_{n=1}^{\infty}\iiiLB( \sum_{k=1}^{n}u_k \iiiRB)^{-1} u_n^2\colon u_k\ge t_k\iiiRB\}
		\le
		8\sum_{n=1}^{\infty}\iiiLB( n\sum_{k=1}^{n}t_k^{2} \iiiRB)^{-1/2}t_n^{2}.	
	\end{gather}
\end{Lm}

\begin{proof}
	For each $ n\in\N $, define the quantities $ U_n $, $ v_n $, $ s_n $:
	\begin{gather}
		\label{2026-06-10_16-30-09}
		U_n
		=
		\sum_{k=1}^{n}u_k,
		\quad
		v_n^2
		=
		\sum_{k=1}^{n}u_k^2,
		\quad
		s_n^2
		=
		\sum_{k=1}^{n}t_k^2.
	\end{gather}
	Note that $ v_n\ge s_n $ for all $ n\in\N $ and prove the lower bound of the sum $ \sum_{n=1}^{\infty}U_n^{-1}u_n^2 $:
	\begin{multline}
		\notag
		\sum_{n=1}^{\infty}\frac{t_n^2}{\sqrt{n}\, s_n}
		\overset{\eqref{2026-06-09_12-14-28}}{\le}
		2\sum_{n=1}^{\infty}\frac{s_n - s_{n-1}}{\sqrt{n}}
		\le
		2\sum_{n=1}^{\infty}\frac{v_n - v_{n-1}}{\sqrt{n}}
		\\
		\notag
		=
		2\sum_{n=1}^{\infty}\frac{v_n^2 - v_{n-1}^2}{\sqrt{n}(v_n + v_{n-1})}
		\le
		2\sum_{n=1}^{\infty}\frac{u_n^2}{\sqrt{n}v_n}
		\le
		2\sum_{n=1}^{\infty}\frac{u_n^2}{U_n}.
	\end{multline}
	
	Next, we assume that right-hand side of~\eqref{2026-06-09_13-52-09} is finite and construct the appropriate sequence $ u_n $ to obtain the upper bound. To do this, we define, for all $ n\in\N $, the quantities $ r_n $ and $ \lambda $:
	\begin{gather}
		\label{2026-05-18_12-40-33}
		r_n
		=
		\sup\{k^{-3/2}s_k,\ k\ge n\},
		\quad
		\lambda
		=
		\sum_{n=1}^{\infty}\frac{t_n^2}{\sqrt{n}\,s_n}
		<
		\infty.
	\end{gather}
	The sequence $ r_n $ is a non-increasing majorant of the sequence $ n^{-3/2}s_n $. Since $ k^{-3/2}s_k\le k^{-1} $, the supremum in \eqref{2026-05-18_12-40-33} is attained, and there exists a strictly increasing function $ \nu\colon\N_0\to\N_0 $ such that
	\begin{gather}
		\label{2026-05-18_12-52-15}
		\nu(0) = 0,
		\quad
		r_n
		=
		r_{\nu(k)}
		=
		\nu(k)^{-3/2}s_{\nu(k)},
		\quad
		\text{where}
		\quad
		k
		=
		\inf\{l\in\N_0\colon \nu(l)\ge n\}.
	\end{gather}
	Note that the function $ \nu $ may not be unique. We now define the numbers $ u_n $:
	\begin{gather}
		\label{2026-05-18_12-52-50}
		u_n
		=
		\max(t_n, nr_n).
	\end{gather}
	Next, we estimate the value $ U_n=\sum_{k=1}^{n}u_k $ from below and establish the following inequalities:
	\begin{gather}
		\label{2026-06-09_15-48-32}
		2U_n
		\overset{\eqref{2026-05-18_12-52-50}}{\ge}
		2\sum_{k=1}^{n}kr_k
		\ge
		2r_n\sum_{k=1}^{n}k
		\ge
		n^2r_n
		\overset{\eqref{2026-05-18_12-40-33}}{\ge}
		\sqrt{n}\,s_n.
	\end{gather}
	We are ready to obtain the desired estimate:
	\begin{gather}
		\notag
		\frac{1}{2}\sum_{n=1}^{\infty}\frac{u_n^2}{U_n}
		\overset{\eqref{2026-05-18_12-52-50}, \eqref{2026-06-09_15-48-32}}{\le}
		\sum_{n=1}^{\infty}\frac{t_n^2}{\sqrt{n}\,s_n}
		+
		\sum_{n=1}^{\infty}\frac{n^2r_n^2}{n^2r_n}
		\overset{\eqref{2026-05-18_12-40-33}}{=}
		\lambda + \sum_{n=1}^{\infty}r_n
		\\
		\notag
		\overset{\eqref{2026-05-18_12-52-15}}{=}
		\lambda + \sum_{k=1}^{\infty}(\nu(k) - \nu(k-1))r_{\nu(k)}
		\overset{\eqref{2026-05-18_12-52-15}}{=}
		\lambda + \sum_{k=1}^{\infty}(\nu(k) - \nu(k-1))\frac{s_{\nu(k)}^2}{\nu(k)^{3/2}s_{\nu(k)}}
		\\
		\notag
		\overset{\eqref{2026-06-10_16-30-09}}{=}
		\lambda + \sum_{k=1}^{\infty}\frac{\nu(k) - \nu(k-1)}{\nu(k)^{3/2}s_{\nu(k)}}\sum_{n=1}^{\nu(k)}t_n^2
		\le
		\lambda + \sum_{n=1}^{\infty}\frac{t_n^2}{\sqrt{n}\,s_n}
		\sum_{\substack{k\in\N\\ \nu(k)\ge n}}\frac{\sqrt{n}(\nu(k) - \nu(k-1))}{\nu(k)^{3/2}}.
	\end{gather}
	Let $ l $ be the smallest natural number such that $ \nu(l)\ge n $. Then
	\begin{gather}
		\notag
		\sum_{\substack{k\in\N\\ \nu(k)\ge n}}\frac{\sqrt{n}(\nu(k) - \nu(k-1))}{(\nu(k))^{3/2}}
		=
		\frac{\sqrt{n}(\nu(l) - \nu(l-1))}{\nu(l)^{3/2}}
		+
		\sum_{\substack{k\in\N\\ \nu(k-1)\ge n}}\frac{\sqrt{n}(\nu(k) - \nu(k-1))}{(\nu(k))^{3/2}}
		\\
		\notag
		\le
		\frac{\sqrt{n}}{\sqrt{\nu(l)}}
		+
		\sqrt{n}\int_{n}^{+\infty}x^{-3/2}\sdiff x
		\le
		1 + 2
		=
		3.
	\end{gather}
	Therefore, we obtain the desired estimate:
	\begin{gather}
		\notag
		\frac{1}{2}\sum_{n=1}^{\infty}\frac{u_n^2}{U_n}
		\le
		\lambda + 3\sum_{n=1}^{\infty}\frac{t_n^2}{\sqrt{n}\,s_n}
		=
		4\lambda.
	\end{gather}
\end{proof}

	\section*{Acknowledgments}
		The author would like to thank V. N. Temlyakov and A. P. Solodov for posing the problem during the school held at the Sirius Mathematical Center. The author also acknowledges the assistance of ChatGPT-5 with setting up computer experiments and ChatGPT-5.5 with detecting the equivalence of Conditions~\ref{2026-06-08_16-57-27} and~\ref{2026-05-17_13-35-30} of Theorem~\ref{2026-05-17_13-33-16}, proving Lemma~\ref{2026-06-08_21-18-46}, and significantly simplifying the lower bound in Lemma~\ref{2026-06-02_12-10-50}. The author is grateful to D. M. Stolyarov for his help with the preparation of the paper.

	\bibliographystyle{amsplain}
	\bibliography{greedy_algorithms}
	
	\bigskip
	
	\noindent
	\textsc{St. Petersburg State University}
	
	\noindent
	\textit{Email address}: \texttt{Novikov.3.14@yandex.ru}
\end{document}